\documentclass[12pt]{amsart}
\usepackage{amssymb,amscd,amsmath}
\usepackage{tikz}\usepackage{tikz-cd}
\usepackage[all]{xy} 
\usepackage{enumerate}

\usepackage{mathabx,epsfig}
\def\acts{\mathrel{\reflectbox{$\righttoleftarrow$}}}

\newtheorem{fact}{Fact}[section]   

\newtheorem{theorem}[fact]{Theorem}
\newtheorem{proposition}[fact]{Proposition}

\newtheorem{remark}[fact]{Remark}
\newtheorem{definition}[fact]{Definition}
\newtheorem{example}[fact]{Example}
\newtheorem{conjecture}[fact]{Conjecture}
\newtheorem{assumption}[fact]{Assumption}

\DeclareMathOperator{\Q}{\mathbb Q}
\DeclareMathOperator{\C}{\mathbb C}

\DeclareMathOperator{\Z}{\mathbb Z}
\DeclareMathOperator\LCn{\Lambda\!^2\!\C^n} \DeclareMathOperator\LCm{\Lambda\!^2\!\C^m}
\DeclareMathOperator\SCn{S^2\!\C^n}
\DeclareMathOperator\GL{GL}
\DeclareMathOperator\Sp{Sp}
\DeclareMathOperator\codim{codim}
\DeclareMathOperator\One{\hbox{\bf 1}}
\DeclareMathOperator\csm{c^{sm}}
\DeclareMathOperator\ssm{s^{sm}}
\DeclareMathOperator\ssmz{s^{sm}_0}
\DeclareMathOperator\csmz{c^{sm}_0}
\DeclareMathOperator\cM{c^{M}}
\DeclareMathOperator\st{\tilde{s}}
\DeclareMathOperator\mC{mC}
\DeclareMathOperator\mS{mS}
\DeclareMathOperator\Gr{Gr}
\DeclareMathOperator\Eu{Eu}
\DeclareMathOperator\Ann{Ann}
\DeclareMathOperator\spa{span}
\DeclareMathOperator\PP{\mathbb P}
\DeclareMathOperator\PS{\mathbb P \Sigma}
\DeclareMathOperator\PV{\mathbb P V}
\DeclareMathOperator{\J}{\mathcal J}
\DeclareMathOperator{\Hom}{Hom}

\title{Characteristic Classes of Symmetric and Skew-symmetric degeneracy loci}

\author{Sutipoj Promtapan}
\address{Department of Mathematics, University of North Carolina at Chapel Hill, USA}
\email{spromtapan@gmail.com}

\author{Rich\'ard Rim\'anyi}
\address{Department of Mathematics, University of North Carolina at Chapel Hill, USA}
\email{rimanyi@email.unc.edu}

\begin{document}

\maketitle

\begin{abstract}
We give two formulas for the Chern-Schwartz-MacPherson class of symmetric and skew-symmetric degeneracy loci. We apply them in enumerative geometry, explore their algebraic combinatorics, and discuss K theory generalizations.  
\end{abstract}

\section{Introduction} Degeneracy loci formulas are universal expressions for the characteristic classes of certain degeneracy loci. The two most widely used such formulas are 
\begin{itemize}
\item the Giambelli-Thom-Porteous formula \cite{Po}
\begin{equation*}\label{eq1}
[\overline{\Sigma}_r]=s_{(r+l)^r},
\end{equation*}
and 
\item formulas of J\'ozefiak-Lascoux-Pragacz and Harris-Tu \cite{JLP, HT, FR1,AF}
\begin{equation*}\label{eq2}
[\overline{\Sigma}^\wedge_r]=s_{r-1,r-2,\ldots,2,1}, \qquad [\overline{\Sigma}^S_r]=2^{r-1}s_{r,r-1,\ldots,2,1}.
\end{equation*}
\end{itemize}
Some explanations are in order. 

\subsection{Degeneracy loci interpretation} First we explain the two formulas above in the language of ``degeneracy loci''. Let $\psi:A^n\to B^{n+l}$ ($l\geq 0$) be a vector bundle map over the base space $M$, and let $\Sigma_r$ be the set of points $x$ in $M$ over which $\psi_x$ has rank $n-r$ (that is, corank $r$). Then under suitable assumption on $M$ and transversality assumption on $\psi$, the above Giambelli-Thom-Porteous formula holds for the {\em fundamental cohomology class} $[\overline{\Sigma}_r]\in H^*(M)$ of $\overline{\Sigma}_r$, where $s_{\lambda_1,\ldots,\lambda_k}=\det(c_{\lambda_i+j-i})_{i,j=1,\ldots,k}$ and $c_i$ is defined by 
\begin{equation}\label{quotient}
1+c_1t+c_2t^2+\ldots=\frac{1+c_1(B)t+c_2(B)t^2+\ldots}{1+c_1(A)t+c_2(A)t^2+\ldots}.
\end{equation}

Now let $A$ be a rank $n$ vector bundle, and $\psi:A^*\to A$ be a {\em skew-symmetric} or {\em symmetric} vector bundle map over the base space $M$, and let $\Sigma^r$ be the set of points $x$ in $M$ over which $\psi_x$ has corank $r$ (in the skew-symmetric case $n-r$ is necessarily even). Then, under suitable assumption on $M$ and transversality assumption on $\psi$, the J\'ozefiak-Lascoux-Pragacz-Harris-Tu formulas hold, where $s_\lambda$ is the same as above, with $c_i$ the $i$th Chern class of $A$. 

\subsection{Equivariant cohomology interpretation}
Consider the $G=\GL_n(\C)\times \GL_{n+l}(\C)$ action on $\Hom(\C^n,\C^{n+l})$ by $(A,B)\cdot X=BXA^{-1}$, and let $\Sigma_r$ be the subset in $\Hom(\C^n,\C^{n+l})$ of matrices of corank $r$. Then the Giambelli-Thom-Porteous formula holds for the equivariant fundamental class of $[\overline{\Sigma}_r]$ in $H^*(BG)$. The classes $c_i$ are as in \eqref{quotient} where $a_i$ and $b_i$ are the Chern classes of the tautological rank $n$ and rank $n+l$ vector bundles over $BG$. 

Similarly, consider the $G=\GL_n(\C)$ action on the set of skew-symmetric or symmetric $n\times n$ matrices by $A\cdot X=A^TXA$ and let $\Sigma^\wedge_r, \Sigma^S_r$ be the set of those of corank $r$. Then  for the {\em $G$-equivariant fundamental classes} $[\overline{\Sigma}^\wedge_r], [\overline{\Sigma}^S_r]$  the J\'ozefiak-Lascoux-Pragacz-Harris-Tu formulas hold in $H^*(BG)$, where $c_i$ is the $i$'th Chern class of the tautological bundle over $BG$. 

\subsection{MacPherson deformation of the fundamental class}
The notion of fundamental class has an inhomogeneous deformation, called Chern-Schwartz-MacPherson class (CSM), denoted 
\[\csm(\Sigma)=\csm(\Sigma\subset M)=[\Sigma]+\text{higher order terms}.\]
The CSM class encodes more geometric and enumerative properties of the singular variety $\Sigma$ than its lowest degree term, the fundamental class. It is also related with symplectic topology and representation theory (through Maulik-Okounkov's notion of ``stable envelope classes'' \cite{MO}), at least in Schubert calculus settings, 
see \cite{RV, FR2, AMSS1}.

The CSM version of the Giambelli-Thom-Porteous formula is calculated in \cite{PP}, see also \cite{FR2, Z}. To give a sample of that result we introduce the Segre-Schwartz-MacPherson class (SSM): $\ssm(\Sigma\subset M)=\csm(\Sigma\subset M)/c(TM)$. This carries the same information as the CSM class, but certain theorems are phrased more elegantly for SSM classes. We have 
\[
\ssm(\Sigma_0\subset \Hom(\C^n,\C^{n+1}))=s_0-s_2+(2s_3+s_{21})+(-3s_4-3s_{31}-s_{211})+\ldots
\]
(Those looking for positivity properties of such expansions may find this formula disappointing, but luckily positivity can be saved, see \cite[Section 1.5]{FR2}.) 

The goal of this paper is to calculate the CSM (or equivalently, the SSM) deformations of the J\'ozefiak-Lascoux-Pragacz-Harris-Tu formulas. The reader is invited to jump ahead and see sample results in Section \ref{algcomb}. 

\subsection{Plan of the paper}
After introducing our geometric settings (the $\LCn$ and $\SCn$ representations) in Section \ref{sec:TheReps}, we recall the notion of Chern-Schwartz-MacPherson class in Section~\ref{sec:CSM}. In particular, first we recall the traditional approach to CSM classes via resolutions, and push-forward, and then we recall the recent development, triggered by Maulik-Okounkov's notion of stable envelopes, claiming that CSM classes are the unique solutions to some interpolation problems. 

We follow the traditional approach in Section \ref{sec:sieveform}, and we follow the interpolation approach in Section \ref{sec:intform}. Both yield to formulas for CSM classes of the orbits of $\LCn$ and $\SCn$. The fact that the formulas obtained in the two approaches are equal is not obvious algebraically from their form. The one obtained from interpolation seems better: it is ``one summation shorter'', also, the other one is an exclusion-inclusion formula (sum of terms with alternating signs), hence it is not obviously suitable for further combinatorial study. 

In Section \ref{algcomb} we make the first steps towards the algebraic combinatorics of the obtained formulas. We discuss stability, normalization, and most importantly positivity properties. The positivity properties can be studied for Schur expansions, or for the more conceptual expansions in terms of $\st_\lambda$ functions. 

In Section \ref{appl} we show sample applications in geometry of the calculated CSM classes. Namely, we focus on the most direct consequences, the Euler characteristics of general linear sections of symmetric and skew-symmetric degeneracy loci. 

Finally, in Section \ref{future} we discuss two natural directions for future study. First we give sample results about the closely related Chern-Mather classes of symmetric and skew-symmetric degeneracy loci. Then we explore the natural K theory analogue of CSM class, the so-called {\em motivic Chern class}. The traditional approach to motivic Chern classes is similar to the traditional approach to CSM classes (roughly speaking, replace the notion of Euler characteristic with that of chi-$y$-genus). Hence the K theory analogs of the results in Section  \ref{sec:sieveform} are promising. However, the interpolation  approach to motivic Chern classes is more sophisticated \cite{FRW2}, hence finding analogs of the results in Section \ref{sec:intform} remains a challenge.

\bigskip

\noindent{\bf Acknowledgment.} The first author was supported by the he Development and Promotion of Science and Technology Talents Project (Royal Government of Thailand scholarship) during his doctoral studies at UNC Chapel Hill. The second author is supported by a Simons Foundation grant. 

\bigskip

\noindent{\bf Notation.} Denote $[n]=\{1,\ldots,n\}$. The set of $r$-element subsets of $[n]$ will be denoted by $\binom{[n]}{r}$. For $I\in \binom{[n]}{r}$ let $\bar{I}=[n]-I$. Varieties are considered over the complex numbers, and cohomology is meant with rational coefficients.

\section{The representations $\LCn$, $\SCn$}\label{sec:TheReps}

Consider the action of $\GL_n(\C)$ on the vector space of skew-symmetric $n\times n$ matrices, and on the 
vector space of symmetric $n\times n$ matrices, by $A\cdot X=A^T X A$. These representations will be denoted by $\LCn$ and $\SCn$ respectively. The orbits of both of these representations are determined by rank. 
\begin{itemize}
\item For $0\leq r\leq n$, $n-r$ even, the orbit of rank $n-r$ (``corank $r$'') matrices in $\LCn$ will be denoted by $\Sigma^\wedge_{n,r}$. For example $X_{n,r}^\wedge=\underbrace{H\oplus \ldots \oplus H}_{(n-r)/2} \oplus \underbrace{0 \oplus \ldots 0}_r\in \Sigma_{n,r}^\wedge$, where $H=\begin{pmatrix}0 & 1 \\ -1 & 0\end{pmatrix}$.  We have $\codim (\Sigma_{n,r}^\wedge\subset \LCn)=\binom{r}{2}$.
\item For $0\leq r\leq n$, the orbit of rank $n-r$ (``corank $r$'') matrices in $\SCn$ will be denoted by $\Sigma_{n,r}^S$. For example $X_{n,r}^S=\underbrace{1\oplus \ldots \oplus 1}_{n-r} \oplus \underbrace{0 \oplus \ldots 0}_r\in \Sigma_{n,r}^S$. We have $\codim (\Sigma_{n,r}^S\subset \SCn)=\binom{r+1}{2}$.
\end{itemize}
In later sections we will approach the geometric study of these orbits by constructing resolutions of their closures, and by studying their stabilizer groups.

\section{Chern-Schwartz-MacPherson classes}\label{sec:CSM}

Deligne and Grothendieck conjectured \cite{Sullivan71} and MacPherson proved \cite{MacPherson1974} the existence of a unique natural transformation $C_*: \mathcal{F}(-) \rightarrow H_*(-)$ from the covariant functor of constructible functions to the covariant functor of Borel-Moore homology, satisfying certain properties. Independently, Schwartz \cite{Schwartz1965} introduced the notion of `obstruction class' (for the extension of stratified radial vector frames over a complex algebraic variety), and later Brasselet and Schwartz \cite{Brasselet-Schwartz1981} proved that $C_*$ and the obstruction class essentially coincide, via Alexander duality. We will study the {\em equivariant} {\em co}homology version of the resulting Chern-Schwartz-MacPherson (CSM) class, due to Ohmoto \cite{Ohmoto2006, Ohmoto2012, Ohmoto2016}.

\subsection{Equivariant CSM class, after MacPherson, Ohmoto}\label{sec:eqCSM}

Let $G$ be an algebraic group acting on the smooth algebraic variety $M$, and let $f$ be a $G$-invariant constructible function $f$ on $M$ (say, to $\Z$). The associated {\em $G$-equivariant Chern-Schwartz-MacPherson class} $\csm(f)$ is an element of $H^*_G(M)=H^*_G(M,\Q)$. 

Before further discussing this notion let us consider a version of it, the $G$-equivariant Segre-Schwartz-MacPherson (SSM) class $\ssm(f)=\csm(f)/c(TM)\in H^{**}_G(M)$.\footnote{Since we divided by the equivariant total Chern class ``$c(TM)=1+$higher order terms'', the SSM class may be non-zero in arbitrarily high degrees: it lives in the completion, as indicated. In the rest of the paper we will not indicate this completion, and write $\csm(f), \ssm(f)\in H^*_G(M)$.} Also, for an invariant (not necessarily closed) subvariety $\Sigma\subset M$ denote $\csm(\Sigma)=\csm(\Sigma\subset M)=\csm(\One_\Sigma)\in H^*_G(M)$, and $\ssm(\Sigma)=\ssm(\Sigma\subset M)=\ssm(\One_\Sigma)\in H^*_G(M)$, where $\One_\Sigma$ is the indicator function of $\Sigma$.

We will sketch Ohmoto's definition below in Remark \ref{OhmotosDef}. For our purposes the following (defining) properties will be sufficient.

\begin{enumerate}[(i)]
\item\label{csm property 1} (additivity) For equivariant constructible functions $f$ and $g$ on $M$ we have 
\[
\csm(f+g) = \csm(f) + \csm(g)\qquad \text{and}\qquad \csm(\lambda\cdot f) = \lambda\cdot \csm(f) \text{ for } \lambda\in \Z.
\]
\item\label{csm property 2} (normalization) For an equivariant proper embedding of a {\em smooth} subvariety $i:\Sigma \subset M$ we have 
\[
\csm(\Sigma \subset M) = i_* c(T\Sigma) \in H^*_G(M).
\]
\item\label{csm property 3} (functoriality) For a $G$-equivariant proper map between smooth $G$-varieties $\eta : Y \rightarrow M$,
\[
\eta_*(c(TY)) = \sum_j j \cdot \csm(M_j)
\]
where $M_j = \{ x \in M : \chi (\eta^{-1}(x)) = j \}$.
\end{enumerate}
The named three properties uniquely define the CSM class (the uniqueness is obvious, the existence is the content of the arguments of MacPherson, Brasselet-Schwartz, Ohmoto). The CSM class, however, satisfies another key property \cite[Theorem 4.2]{Ohmoto2006}, \cite[Proposition 3.8]{Ohmoto2016}:
\begin{enumerate}[(i)]
\item[(iv)]\label{csm property 4} Let $\Sigma \subset M$ be a closed invariant subvariety with an invariant Whitney stratification. For an equivariant map between smooth manifolds $\eta : Y \rightarrow M$ that is transversal to the strata of $\Sigma$, we have
\[
\ssm(\eta^{-1}(\Sigma)) = \eta^*(\ssm(\Sigma)).
\]
\end{enumerate}

\begin{remark}\rm
The most natural characteristic class, the fundamental class $$[\Sigma \subset M] \in H^{\codim(\Sigma\subset M)}(M)$$ of a closed subvariety $\Sigma\subset M$ behaves nicely with respect to both push-forward and pull-back. The CSM ``deformation'' of the notion of fundamental class is forced to behave nicely with respect to push-forward (see axiom \eqref{csm property 3}). Yet, it is rather remarkable that it remains well-behaving with respect to pull-back as well (see property (iv)).
\end{remark}

We called the CSM  class a `deformation' of the fundamental class because we also have \cite[Section 4.1]{Ohmoto2006}:
\begin{enumerate}[(i)]
\item[(v)]\label{csm property 5} For a subvariety $\Sigma\subset M$ the lowest degree term of $\csm(\Sigma\subset M)$ is $[\Sigma\subset M]$. Terms of degree higher than $\dim M$ are 0. The integral of (the term of degree $\dim M$ of) $\csm(\Sigma \subset M)$ is the topological Euler characteristic of $\Sigma$.
\end{enumerate}

\begin{remark}\rm \label{OhmotosDef} We have set up the CSM classes in cohomology, but their natural habitat is homology. Following MacPherson, Ohmoto defines them by first proving the existence and uniqueness of a natural transformation $C_*^G : \mathcal{F}^G(-) \rightarrow H_*^G(-)$ from the abelian group of $G$-equivariant constructible functions to the $G$-equivariant homology (some non-totally-trivial definitions are needed to make this work!), satisfying axioms analogous to \eqref{csm property 1}--\eqref{csm property 3} above. Then the cohomology version considered in this paper is obtained by composing $C_*^G$ with homology push-forward to the ambient space, and Poincar\'e-duality in the smooth ambient space.
\end{remark}

\subsection{Interpolation characterization of CSM classes}
Under certain circumstances, equivariant CSM classes are also determined by a set of interpolation properties \cite{RV, FR2}.

Consider a linear representation of the algebraic group $G$ on the vector space $V$. For an orbit $\Sigma$, and $x \in \Sigma$ let $G_{\Sigma}\leq G$ be the stabilizer subgroup of $x$. Let $T_{\Sigma} = T_x\Sigma$ be the tangent space of $\Sigma$ at $x$, and $N_{\Sigma} = T_xV / T_{\Sigma}$, both $G_\Sigma$-representations (these  definitions don't depend on the choice of $x$ in $\Sigma$ up to natural isomorphisms). We will use the $G_\Sigma$ equivariant Euler and total Chern classes of these representations. Note that the inclusion $G_{\Sigma} \leq G$ induces a map $\phi_{\Sigma} : H^*(BG) \rightarrow H^*(BG_{\Sigma})$---which is also independent of the choice of $x\in \Sigma$.

\begin{assumption} \label{assume}
We assume that the representation $G\acts V$ has finitely many orbits, the orbits are cones, and that the Euler class $e(N_{\Sigma})\not=0$ for all $\Sigma$.
\end{assumption}

\begin{theorem}[\cite{FR2}]\label{Interpolation axiom}
Under Assumption \ref{assume} the CSM class of the orbit $\Sigma$ is uniquely determined by the conditions
\begin{enumerate}
\item\label{axiom 1} $\phi_{\Sigma}(\csm(\Sigma)) = c(T_{\Sigma})e(N_{\Sigma})$ in $H^*(BG_{\Sigma})$;
\item\label{axiom 2} for any orbit $\Omega$, $c(T_{\Omega})$ divides $\phi_{\Omega}(\csm(\Sigma))$ in $H^{*}(BG_{\Omega})$;
\item\label{axiom 3} for any orbit $\Omega \neq \Sigma$, $\deg(\phi_{\Omega}(\csm(\Sigma))) < \deg(c(T_{\Omega})e(N_{\Omega}))$.
\end{enumerate}
\end{theorem}

\noindent 
It also follows from these conditions that for an orbit $\Omega \not\subset \overline{\Sigma}$ we have $\phi_{\Omega}(\Sigma) = 0$.


\section{Sieve formula for the CSM classes}\label{sec:sieveform}

Our goal is to calculate the $\GL_n(\C)$-equivariant CSM (or SSM) class of the orbits of $\LCn$ and $\SCn$. They are classes in
\[
H^*_{\GL_n(\C)}(\LCn)=H^*_{\GL_n(\C)}(\SCn)=H^*(B\!\GL_n(\C))=\Q[\alpha_1,\ldots,\alpha_n]^{S_n},
\]
where $\alpha_i$'s are the Chern roots of the tautological $n$-bundle over $B\!\GL_n(\C)$. In the whole paper $c_k$ will denote the $k$'th Chern class of that bundle, ie. the $k$'th elementary symmetric polynomial of the $\alpha_i$'s.

In this section we make calculations using ``traditional methods'', and achieve an exclusion-inclusion type formula (Theorems \ref{Thm ssm = sum of phi s-sym}, \ref{Thm ssm = sum of phi sym}), then in the next section we solve the relevant interpolation problem and find improved formulas.

\subsection{Fibered resolution}\label{fibred}

Consider a $G$-representation $V$, and an invariant closed subvariety $\Sigma \subset V$. The $G$-equivaraint map $\eta : \widetilde{\Sigma} \rightarrow V$ is called a fibered resolution of $\Sigma$ if there exists a $G$-equivariant commutative diagram
\begin{center}
\begin{tikzcd}
\widetilde{\Sigma} \arrow[rd] \arrow[r,"i"] \arrow[rr,out=30,in=150,"\eta"] & K \times V \arrow[d,"\pi_K"] \arrow[r,"\pi_V"] & V \\
 & K
\end{tikzcd}
\end{center}
where $\eta$ is a resolution of singularities of $\Sigma$, $\pi_V$ is the projection to $V$, $\pi_K$ is the projection to $K$, the map $\widetilde{\Sigma} \rightarrow K$ is a $G$-vector bundle over a smooth compact $G$-variety, $i$ is a $G$-equivariant embedding of vector bundles, and $\eta = \pi_V \circ i$. Let $\nu = (K \times V \rightarrow K)/(\widetilde{\Sigma} \rightarrow K)$ be the $G$-equivariant quotient bundle over $K$. A pullback of the bundle $\nu$ to $\widetilde{\Sigma}$ is the normal bundle of the embedding $i : \widetilde{\Sigma} \rightarrow K \times V$. Define
\begin{equation}\label{Phidef}
\Phi_{\Sigma} := 
\frac{\eta_*(c(T\widetilde{\Sigma}))}{c(V)} = \eta_*\left( \frac{c(T\widetilde{\Sigma})}{c(V)}\right)=\eta_*(c(-\nu) c(TK))=\int_K e(\nu)c(-\nu)c(TK).
\end{equation}
The equality of the displayed expressions is detailed in \cite[Section 10.1]{FR2}.

The significance of the class $\Phi_\Sigma$ is that, on the one hand, one can write formulas for it (due to its last displayed expression), and, on the other hand, the CSM class of $\Sigma$ is a linear combination of $\Phi$-classes of some varieties contained in $\Sigma$. 

\subsection{Sieve formula for SSM classes of orbits of $\LCn$}\label{sec:sieveLambda}

Elements $X\in \LCn$ will be identified with skew-symmetric bilinear forms on $\C^{n*}$, and in turn, with skew-symmetric linear maps $\C^{n*}\to \C^n$, without further notation.
For $0\leq r\leq n$, $n-r$ even, define 
\[
\widetilde{\Sigma}^{\wedge}_{n,r} := \{ (W,X) \in \Gr_r(\C^{n*}) \times \LCn, X|_{W} = 0 \},
\] 
and consider the diagram 
\begin{center}
\begin{tikzcd}
\widetilde{\Sigma}^{\wedge}_{n,r} \arrow[rd] \arrow[r,"i"] \arrow[rr,out=30,in=150,"\eta"] & \Gr_r(\C^{n*}) \times \LCn \arrow[d,"\pi_1"] \arrow[r,"\pi_2"] & \LCn \\
 & \Gr_r(\C^{n*}),
\end{tikzcd}
\end{center}
where $i$ is the inclusion, $\pi_1, \pi_2$ are projection maps on the first and second coordinates, respectively. This diagram is a fibered resolution of $\overline{\Sigma}^{\wedge}_{n,r}$. 
Consider the corresponding $\Phi$-class (see~\eqref{Phidef})
\[
\Phi^\wedge_{n,r} = \int_{\Gr_r(\C^{n*})} e(\nu)c(-\nu)c(T\Gr_r(\C^{n*})),
\]
where $\nu$ is the quotient bundle $(\Gr_r(\C^{n*}) \times \LCn \to \Gr_r( \C^{n*}))/(\widetilde{\Sigma}^{\wedge}_{n,r} \to \Gr_r(\C^{n*}))$.

\begin{proposition}\label{Proposition Phi formula}
For $0 \leq r \leq n$, $n-r$ even, we have 
\begin{equation}\label{philoc}
\Phi^{\wedge}_{n,r} = \sum_{\substack{ I \subset [n] \\ \mid I \mid = r }} \left( \prod_{i<j \in I}\frac{\alpha_i + \alpha_j}{1+\alpha_i+\alpha_j} \prod_{i \in I} \prod_{j \in \bar{I}} \frac{(\alpha_i+\alpha_j)(1-\alpha_j+\alpha_i)}{(1+\alpha_i+\alpha_j)(-\alpha_j+\alpha_i)} \right).
\end{equation}
\end{proposition}
\begin{proof}
The fiber of the bundle $\widetilde{\Sigma}^{\wedge}_{n,r} \to \Gr_r(\C^{n*})$ over $W\in \Gr(\C^{n*})$ is $\{X\in \LCn: X|_W=0\}=\Lambda^2(\Ann(W))$ where $\Ann(W)=\{v\in \C^n: \phi(v)=0 \text{ for all } \phi\in W\} \in \Gr_{n-r}(\C^n)$. Hence the bundle $\widetilde{\Sigma}^{\wedge}_{n,r} \to \Gr_r(\C^{n*})$ is $\Lambda^2 (Q^*)$, where $0 \rightarrow S \rightarrow \C^{n*} \rightarrow Q \rightarrow 0$  is the tautological exact sequence of bundles over $\Gr_r(\C^{n*})$. 

Hence we have
\[
\nu=\LCn - \widetilde{\Sigma}^{\wedge}_{n,r} = 
\left( 
\Lambda^2 (S^*) \oplus \Lambda^2 (Q^*) \oplus (S^* \otimes Q^*)
\right) 
-
\Lambda^2 (Q^*)
=
\Lambda^2 (S^*) \oplus (S^*\otimes Q^*).
\]
Let $\delta_1,\dots,\delta_r$ be the Chern roots of the bunddle $S$, and let $\omega_1,\dots,\omega_{n-r}$ be the Chern roots of $Q$. Then 
\begin{align*}
e(\nu) &= \prod_{1 \leq i < j \leq r} (-\delta_i - \delta_j) \prod_{i=1}^{r} \prod_{j=1}^{n-r} (-\omega_j-\delta_i), \\
c(\nu) &= \prod_{1 \leq i < j \leq r} (1- \delta_i - \delta_j) \prod_{i=1}^{r} \prod_{j=1}^{n-r} (1-\omega_j-\delta_i), \\
c(T\Gr_r(\C^{n*})) &= \prod_{i=1}^{r} \prod_{j=1}^{n-r} (1+\omega_j-\delta_i),
\end{align*}
and from the definition of $\Phi^\wedge_{n,r}$ we obtain
\[
\Phi^\wedge_{n,r} = \int_{\Gr_r(\C^{n*})} \prod_{1 \leq i < j \leq r} \frac{-\delta_i-\delta_j}{1-\delta_i-\delta_j} \prod_{i=1}^{r} \prod_{j=1}^{n-r} \frac{(-\omega_j-\delta_i)(1+\omega_j-\delta_i)}{1-\omega_j-\delta_i}. 
\]
The equivairant localization formula for this integral in exactly \eqref{philoc}.
\end{proof}

\begin{proposition}\label{Proposition Phi = sum SSM}
For $0 \leq r \leq n$, $n-r$ even, we have
\begin{align*}
\Phi^\wedge_{n,r} &= \sum_{i=0}^{\frac{n-r}{2}} \binom{r + 2i}{r} \ssm(\Sigma^\wedge_{n,r+2i}) \\
&= \sum_{i=0}^{\frac{n-r}{2}} \left( \binom{r+2i}{r} - \binom{r+2i-2}{r} \right) \ssm(\overline{\Sigma}^\wedge_{n,r+2i}).
\end{align*}
\end{proposition}
\begin{proof}
The closure $\overline{\Sigma}^\wedge_{n,r} = \Sigma^\wedge_{n,r} \cup \Sigma^\wedge_{n,r+2} \cup \dots \cup \Sigma^\wedge_{n,n}$ is the image of $\eta$. For each $r \leq k \leq n$ and $n-k$ even, a preimage of each point in $\Sigma^\wedge_{n,k}$ is isomorphic to the space $\Gr_r(\mathbb{C}^{k*})$. Note that the Euler characteristic of the Grassmannian of all $r$-dimensional subspaces of a $k$-dimensional vector space over $\mathbb{C}$ is $\binom{k}{r}$. Using property \eqref{csm property 3} of CSM classes in Section \ref{sec:eqCSM}, we have 
\[
\eta_*(c(T\widetilde{\Sigma}^\wedge_{n,r})) = \sum_{i=0}^{\frac{n-r}{2}} \binom{r + 2i}{r} \csm(\Sigma^\wedge_{n,r+2i}).
\]
Dividing both sides by $c(\LCn)$ proves the first equality of the proposition. 

By the additivity property of SSM classes (see Section \ref{sec:eqCSM} \eqref{csm property 1}), we have
\begin{align*}
\sum_{i=0}^{\frac{n-r}{2}} \binom{r + 2i}{r} \ssm(\Sigma^\wedge_{n,r+2i}) &= \sum_{i=0}^{\frac{n-r}{2}} \binom{r + 2i}{r} \left( \ssm(\overline{\Sigma}^\wedge_{n,r+2i}) - \ssm(\overline{\Sigma}^\wedge_{n,r+2i+2}) \right) \\
&= \sum_{i=0}^{\frac{n-r}{2}} \left( \binom{r+2i}{r} - \binom{r+2i-2}{r} \right) \ssm(\overline{\Sigma}^\wedge_{n,r+2i}),
\end{align*}
which completes the proof.
\end{proof}

Proposition \ref{Proposition Phi = sum SSM} expresses the sought SSM classes as linear combinations of the $\Phi$-classes. Inverting the matrix of the coefficients of these linear combinations will therefore express the SSM classes as linear combinations of the $\Phi$-classes.

\begin{definition}\label{Def Euler number}
Define the Euler numbers $E_n$ by 
\[
\frac{1}{\cosh(x)} = \sum_{n=0}^{\infty} \frac{E_n}{n!}x^n.
\]
\end{definition}

For odd $n$ the number $E_n$ is zero. For even $n$ Euler numbers form an alternating sequence: $E_0=1, E_2=-1, E_4=5, E_6=-61, E_8=1385, E_{10}=-50512, \ldots$. For explicit formulas for the Euler numbers see e.g. \cite{Wei2015} and references therein. 

\begin{proposition}\label{Proposition Inverse matrix}
The inverse of the triangular matrix $\left( \binom{2j}{2i} \right)_{0 \leq i,j \leq m}$ is $\left( \binom{2j}{2i}E_{2j-2i} \right)_{0 \leq i,j \leq m}$, and the inverse of the triangular matrix $\left( \binom{2j+1}{2i+1} \right)_{0 \leq i,j \leq m}$ is $\left( \binom{2j+1}{2i+1}E_{2j-2i} \right)_{0 \leq i,j \leq m}$.
\end{proposition}

\begin{proof} 
The product of $\left( \binom{2j}{2i} \right)_{0 \leq i,j \leq m}$ and $\left( \binom{2j}{2i}E_{2j-2i} \right)_{0 \leq i,j \leq m}$ is upper triangular. For $i\leq j$ its $(i,j)$'th entry is
\[
\sum_{k=i}^j\binom{2k}{2i} \binom{2j}{2k} E_{2j-2k}=\binom{2j}{2i} \sum_{k=i}^j \binom{2j-2i}{2j-2k}E_{2j-2k}=\begin{cases}
1 & i=j \\ 0 & i<j,\end{cases}
\]
where the last equality follows from the defining equation
\[
\left(1+\frac{t^2}{2!}+\frac{t^4}{4!}+\ldots\right)\left(1+\frac{E_2t^2}{2!}+\frac{E_4t^4}{4!}+\ldots\right)=1.
\]
The proof of the other statement is similar.
\end{proof}
For example
\begin{multline*}
\begin{pmatrix}
1 & 1 & 1 & 1 \\
0 & 1 & 6 & 15 \\
0 & 0 & 1 & 15 \\
0 & 0 & 0 & 1 
\end{pmatrix}^{-1}
=
\begin{pmatrix}
\binom{0}{0} & \binom{2}{0} & \binom{4}{0} & \binom{6}{0} \\
0             & \binom{2}{2} & \binom{4}{2} & \binom{6}{2} \\
0             & 0             & \binom{4}{4} & \binom{6}{4} \\
0             & 0             & 0             & \binom{6}{6} 
\end{pmatrix}^{-1}
= \\
\begin{pmatrix}
\binom{0}{0}E_0 & \binom{2}{0}E_2 & \binom{4}{0}E_4 & \binom{6}{0}E_6 \\
0             & \binom{2}{2}E_0 & \binom{4}{2}E_2 & \binom{6}{2}E_4 \\
0             & 0             & \binom{4}{4}E_0 & \binom{6}{4}E_2 \\
0             & 0             & 0             & \binom{6}{6}E_0 
\end{pmatrix}
=
\begin{pmatrix}
1 & -1 & 5 & -61 \\
0 & 1 & -6 & 75 \\
0 & 0 & 1 & -15 \\
0 & 0 & 0 & 1 
\end{pmatrix}.
\end{multline*}

\begin{theorem}\label{Thm ssm = sum of phi s-sym}
For $0 \leq r \leq n$, $n-r$ even, we have
\begin{equation}\label{eq:ssmPhi}
\ssm(\Sigma^\wedge_{n,r}) = \sum_{i=0}^{\frac{n-r}{2}} \binom{r+2i}{r} E_{2i} \cdot \Phi^\wedge_{n,r+2i} .
\end{equation}
\end{theorem}
\begin{proof}
This is a consequence of Proposition \ref{Proposition Phi = sum SSM} and Proposition \ref{Proposition Inverse matrix}.
\end{proof}

This theorem, together with the expression \eqref{philoc} for the $\Phi$-classes is our first formula for the SSM classes $\Sigma^\wedge_{n,r}$. We call it the sieve formula, because the coefficients in the summation \eqref{eq:ssmPhi} have alternating signs.

\begin{example} \rm
For $n=2$ we have
\[
\Phi^\wedge_{2,0} = 1, \qquad\qquad \Phi^\wedge_{2,2} = \frac{\alpha_1+\alpha_2}{1+\alpha_1+\alpha_2}
= c_1 - c_1^2 + c_1^3 - c_1^4 + \dots,
\]
and therefore 
\begin{align*}
\ssm(\Sigma^\wedge_{2,0}) & = \Phi^\wedge_{2,0} - \Phi^\wedge_{2,2} = 1 - c_1 + c_1^2 - c_1^3 + c_1^4 - \dots \\
\ssm(\Sigma^\wedge_{2,2}) &= \Phi^\wedge_{2,2} = c_1 - c_1^2 + c_1^3 - c_1^4 + \dots
\end{align*}
\noindent For $n=3$ we have
\begin{align*}
\Phi^\wedge_{3,1} &= 1 + (2c_1c_2-2c_3) + (-4c_1^2c_2+4c_1c_3) + (4c_1^3c_2+2c_1c_2^2-4c_1^2c_3-2c_2c_3) + \\
&  (-10c_1^2c_2^2+12c_1c_2c_3-2c_3^2) + (-8c_1^5c_2+24c_1^3c_2^2+2c_1c_2^3+8c_1^4c_3 -32c_1^2c_2c_3+8c_1c_3^2) + \dots, \\
\Phi^\wedge_{3,3} &= (c_1c_2-c_3) + (-2c_1^2c_2+2c_1c_3) + (2c_1^3c_2+c_1c_2^2-2c_1^2c_3-c_2c_3) +\\
&  (-5c_1^2c_2^2+6c_1c_2c_3-c_3^2) + (-4c_1^5c_2+12c_1^3c_2+c_1c_2^3+4c_1^4c_3 -16c_1^2c_2c_3-c_2^2c_3+4c_1c_3^2) + \dots,
\end{align*}
and 
\[
\ssm(\Sigma^\wedge_{3,1}) = \Phi^\wedge_{3,1} - 3\Phi^\wedge_{3,3}, \qquad\qquad \ssm(\Sigma^\wedge_{3,3}) = \Phi^\wedge_{3,3}.
\]
\end{example}

\subsection{Sieve formula for SSM classes of orbits of $\SCn$}

Arguments analogous to those in Section \ref{sec:sieveLambda} give the following theorem, we leave the details to the reader.

\begin{theorem}\label{Thm ssm = sum of phi sym}
For $0 \leq r \leq n$, we have
\begin{gather*}
\ssm(\Sigma^S_{n,r}) = \sum_{i=0}^{n-r} (-1)^i \binom{r+i}{r} \Phi^S_{n,r+i}, \\
\ssm(\overline{\Sigma}^S_{n,r}) = \sum_{i=0}^{n-r} (-1)^i \binom{r+i-1}{r-1} \Phi^S_{n,r+i}.
\end{gather*}
where 
\[
\Phi^{S}_{n,r} = \sum_{\substack{ I \subset [n] \\ \mid I \mid = r }} \left( \prod_{i \leq j \in I}\frac{\alpha_i + \alpha_j}{1+\alpha_i+\alpha_j} \prod_{i \in I} \prod_{j \in \bar{I}} \frac{(\alpha_i+\alpha_j)(1-\alpha_j+\alpha_i)}{(1+\alpha_i+\alpha_j)(-\alpha_j+\alpha_i)} \right).\qed
\] 
\end{theorem}

The fact that the sieve coefficients for symmetric loci are simple binomial coefficients, not Euler numbers are due to the fact that $\Sigma^S_{n,r}$ orbits exist for all $r$ independent of parity, so at a certain point in the argument we need to invert the matrix of {\em all} binomial coefficients, not only the even ones.

\begin{example}\rm
For $n=2$, we have
\begin{align*}
\Phi^S_{2,0} &= 1, \\
\Phi^S_{2,1} &= \frac{2(\alpha_1+\alpha_2)(1+\alpha_1+\alpha_2+4\alpha_1\alpha_2)}{(1+2\alpha_1)(1+2\alpha_2)(1+\alpha_1+\alpha_2)} \\
&= 2c_1 - 4c_1^2 + 8c_1^3 + (-16c_1^4+8c_1^2c_2) + (32c_1^5-40c_1^3c_2) + \dots, \\
\Phi^S_{2,2} &= \frac{4\alpha_1\alpha_2(\alpha_1+\alpha_2)}{(1+2\alpha_1)(1+2\alpha_2)(1+\alpha_1+\alpha_2)} \\
&= 4c_1c_2 - 12c_1^2c_2 + (28c_1^3c_2-16c_1c_2^2) + \dots,
\end{align*}
and 
\begin{align*}
\ssm(\Sigma^S_{2,0}) &= \Phi^S_{2,0} - \Phi^S_{2,1} + \Phi^S_{2,2} =
1 - 2c_1 + 4c_1^2 + (-8c_1^3+4c_1c_2) + (16c_1^4-20c_1^2c_2) + \dots, \\
\ssm(\Sigma^S_{2,1}) &= \Phi^S_{2,1} - 2\Phi^S_{2,2} =
2c_1 - 4c_1^2 + (8c_1^3-8c_1c_2) + (-16c_1^4+32c_1^2c_2) + \dots, \\
\ssm(\Sigma^S_{2,2}) &= \Phi^S_{2,2} =
 4c_1c_2 - 12c_1^2c_2 + (28c_1^3c_2-16c_1c_2^2) + \dots.
\end{align*}
\end{example}

\section{Interpolation formula for CSM classes}\label{sec:intform}
In Section \ref{sec:W} we define some functions, that we call $W$-functions because of their vague similarity to {\em weight functions} in \cite{RTV}. Then in Section \ref{sec:CSMW} we show that they represent CSM classes.

\subsection{The W-functions}\label{sec:W}

For $I \subset [n]$ let $\boldsymbol{\alpha}_{I}=\{\alpha_{i} : i \in I \}$. A permutation $\tau\in S_k$ acts on a  rational function $f(\alpha_1,\ldots,\alpha_k)$ by $(\tau \cdot f)(\alpha_1,\alpha_2,\dots,\alpha_k) = f(\alpha_{\tau(1)},\alpha_{\tau(2)},\dots,\alpha_{\tau(k)})$.

\begin{definition}\label{def local skew-symm weight fcn}
For $0 \leq r \leq n$, $n-r$ even, define the skew-symmetric W-function
\[
W^{\wedge}_{n,r}(\boldsymbol{\alpha}_{[n]}) = \sum_{\substack{\vert I \vert = r \\ I \subset [n] }} \left( W^{\wedge}_{n-r}(\boldsymbol{\alpha}_{\bar{I}}) \prod_{i<j \in I}(\alpha_{i}+\alpha_{j}) \prod_{i \in I} \prod_{j \in \bar{I}} \frac{(\alpha_{i}+\alpha_{j})(1+\alpha_{i}+\alpha_{j})}{\alpha_{i}-\alpha_{j}} \right)
\]
where
\begin{align*}
W^{\wedge}_{k}(\boldsymbol{\alpha}_{[k]}) &= \frac{1}{2^{\frac{k}{2}} \cdot \left( \frac{k}{2} \right) ! } \sum_{ \tau \in S_{k} } \tau \left( \prod_{ 1 \leq i < j \leq k }  \frac{(1+\alpha_{i}+\alpha_{j})(\alpha_{i}+\alpha_{j})}{\alpha_{i}-\alpha_{j}}  \prod_{i=1}^{\frac{k}{2}}  \frac{\alpha_{2i-1}-\alpha_{2i}}{(\alpha_{2i-1}+\alpha_{2i})(1+\alpha_{2i-1}+\alpha_{2i})} \right).
\end{align*}
\end{definition}

Despite its appearance, $W^{\wedge}_{n,r}$ is a polynomial (denominators cancel), it is in fact an integer coefficient symmetric polynomial in $\boldsymbol{\alpha}_{[n]}$, of highest degree term of degree $\frac{1}{2}(n^2-2n+r)$. The function $W^{\wedge}_k$ can be rewritten as
\[
W^{\wedge}_{k}(\boldsymbol{\alpha}_{[k]}) = \frac{1}{\left( \frac{k}{2} \right) !} \sum_{ \substack{ \vert I_{1} \vert =  \dots = \vert I_{ \frac{k}{2} } \vert = 2 \\ I_{1} \cup \dots \cup I_{ \frac{k}{2} } = [k] } } \left( \prod_{ 1 \leq i < j \leq \frac{k}{2} } \prod_{\substack{i' \in I_{i} \\ j' \in I_{j}}} \frac{(\alpha_{i'}+\alpha_{j'})(1+\alpha_{i'}+\alpha_{j'})}{\alpha_{i'}-\alpha_{j'}} \right).
\]

\begin{example} \rm\label{ex:weight4}
We have $W^{\wedge}_{2,0} = 1, W^{\wedge}_{2,2} = \alpha_{1}+\alpha_{2}=c_1$, and
\begin{align*}
W^\wedge_{3,1} &= 1 + 2\alpha_1 + 2\alpha_2 + 2\alpha_3 + \alpha_1^2 + \alpha_2^2 +\alpha_3^2 +3\alpha_1\alpha_2 +3\alpha_1\alpha_3 +3\alpha_2\alpha_3=  1 + 2c_1 + c_1^2 + c_2,\\
W^\wedge_{3,3} &= \alpha_1^2\alpha_2 + \alpha_1^2\alpha_3 +\alpha_2^2\alpha_3 + \alpha_1\alpha_2^2 + \alpha_1\alpha_3^2 +\alpha_2\alpha_3^2 + 2\alpha_1\alpha_2\alpha_3=c_1c_2-c_3,\\
W^{\wedge}_{4,0}  &= 1 + 2c_{1} + c_{1}^{2} + 2c_{2} + 2c_{1}c_{2} +c_{2}^{2} + c_{1}c_{3} - 4c_{4}, \\
W^{\wedge}_{4,2}  &= c_{1} + 2 c_{1}^{2} + c_{1}^{3} + 2c_{1}c_{2} + 2c_{1}^{2}c_{2} + c_{1}c_{2}^{2} + c_{1}^{2}c_{3} - 4c_{1}c_{4}, \\
W^{\wedge}_{4,4}  &= c_{1}c_{2}c_{3} - c_{1}^{2}c_{4} - c_{3}^{2}.
\end{align*}
\end{example}

\begin{definition}\label{def local symm weight fcn}
For $0 \leq r \leq n$, define the symmetric W-function
\[
W^{S}_{n,r}(\boldsymbol{\alpha}_{[n]}) = \sum_{\substack{\vert I \vert = r \\ I \subset [n] }} \left( W^{S}_{n-r}(\boldsymbol{\alpha}_{\bar{I}}) \prod_{i \leq j \in I} (\alpha_{i}+\alpha_{j}) \prod_{i \in I} \prod_{j \in \bar{I}} \frac{(\alpha_{i}+\alpha_{j})(1+\alpha_{i}+\alpha_{j})}{\alpha_{i}-\alpha_{j}} \right)
\]
where
\begin{align*}
W^{S}_{k}(\boldsymbol{\alpha}_{[k]}) &= \frac{1}{\lfloor{\frac{k}{2}}\rfloor !}\sum_{\tau \in S_{k}} \tau \left( \prod_{1\leq i < j \leq k} \frac{(1+\alpha_{i}+\alpha_{j})(\alpha_{i}+\alpha_{j})}{\alpha_{i}-\alpha_{j}} \prod^{\lfloor{\frac{k}{2}}\rfloor}_{i=1} \frac{-\alpha_{2i}(1+2\alpha_{2i-1})(1-\alpha_{2i-1}+\alpha_{2i})}{(\alpha_{2i-1}+\alpha_{2i})(1+\alpha_{2i-1}+\alpha_{2i})} \right).
\end{align*}
\end{definition}

The function $W^{S}_{n,r}$ is also a symmetric polynomial in $\boldsymbol{\alpha}_{[n]}$, of highest degree component of degree $\frac{n(n+1)}{2}-\lfloor{\frac{n-r+1}{2}}\rfloor$.

\begin{example} \rm
We have
$W^{S}_{2,0} = 1+\alpha_{1}+\alpha_{2}+4\alpha_{1}\alpha_{2}=1+c_1+4c_2$, 
$W^{S}_{2,1} = 2\alpha_{1} + 2\alpha_{2} + 2\alpha_{1}^{2} + 4\alpha_{1}\alpha_{2} + 2\alpha_{2}^{2}
=2c_1+2c_1^2$, $W^{S}_{2,2}= 4\alpha_{1}^{2}\alpha_{2} + 4\alpha_{1}\alpha_{2}^{2}=4c_1c_2$.
\end{example}

\subsection{CSM classes in $\LCn$ and $\SCn$ as W-functions}\label{sec:CSMW}

\begin{theorem}\label{localization thm}
We have $\csm(\Sigma^{\wedge}_{n,r})= W_{n,r}^{\wedge}$ and $\csm(\Sigma^{S}_{n,r}) = W_{n,r}^{S}$.
\end{theorem}

\begin{proof}
We will show that $W^\wedge_{n,r}$ satisfies the three properties in Theorem~\ref{Interpolation axiom} for $\LCn$; along the way we will see that the representation $\LCn$ satisfies Assumption \ref{assume}, hence the verification of the three properties proves $\csm(\Sigma^{\wedge}_{n,r})= W_{n,r}^{\wedge}$.

Choosing the matrix $X^{\wedge}_{n,r}$ (see Section \ref{sec:TheReps}) as the representative of the orbit $\Sigma^{\wedge}_{n,r}$ we can read the following data:
\begin{itemize} 
\item The stabilizer group $\GL_n(\mathbb{C})_{\Sigma^{\wedge}_{n,r}}$ deformation retracts to $\Sp(n-r,\mathbb{C}) \times \GL_{r}(\mathbb{C})$. The maximal torus of $\Sp(n-r,\mathbb{C}) \times \GL_{r}(\mathbb{C})$ embeds into the maximal torus of $\GL_n(\C)$ by
$(s_1,s_2,\ldots,s_{(n-r)/2},a_{n-r+1},\ldots,a_n) \mapsto 
(s_1,-s_1,s_2,-s_2,\ldots,s_{(n-r)/2},-s_{(n-r)/2},a_{n-r+1},\ldots,a_n)$
and hence the map $\phi_{\Sigma^{\wedge}_{n,r}}$ on Chern roots is
\[
(\alpha_1,\ldots,\alpha_{n-r},\alpha_{n-r+1},\ldots,\alpha_n) \mapsto
(\sigma_1,-\sigma_1,\ldots,\sigma_{(n-r)/2},-\sigma_{(n-r)/2}, \alpha_{n-r+1},\ldots,\alpha_n).
\]
Equivalently, we have 
\[
\phi_{\Sigma^{\wedge}_{n,r}}: 
\prod_{i=1}^{n}(1+\alpha_i)
\mapsto
\prod_{i=1}^{\frac{n-r}{2}}(1-\sigma_i^2) \prod_{i=n-r+1}^{n}(1+\alpha_i ).
\]
\item  
$T_{\Sigma^\wedge_{n,r}}=\spa( e_i\otimes e_j-e_j\otimes e_i: 1\leq i<j\leq n, i\leq n-r)$, $N_{\Sigma^\wedge_{n,r}}=\spa( e_i\otimes e_j-e_j\otimes e_i: n-r+1\leq i<j\leq n)$ and hence
\begin{align*}
c(T_{\Sigma^{\wedge}_{n,r}}) &= \prod_{1 \leq i < j \leq \frac{n-r}{2}} (1 \pm \sigma_{i} \pm \sigma_{j}) \prod_{i=1}^{\frac{n-r}{2}} \prod_{j=n-r+1}^{n} (1 \pm \sigma_{i} + \alpha_{j}) , \\
e(N_{\Sigma^{\wedge}_{n,r}}) &= \prod_{n-r+1 \leq i < j \leq n} (\alpha_{i} + \alpha_{j}),
\end{align*}
where we used the short-hand notations $(x \pm \sigma):=(x + \sigma)(x-\sigma)$ and $(1\pm \sigma_i \pm \sigma_j):=(1+\sigma_i+\sigma_j)(1+\sigma_i-\sigma_j)(1-\sigma_i+\sigma_j)(1-\sigma_i-\sigma_j)$. 
\end{itemize}
We see that $e(N_{\Sigma^{\wedge}_{n,r}})\not=0$ which proves that the representation $\LCn$ satisfies Assumption \ref{assume}.

Now we need to prove that the function $W^{\wedge}_{n,r}$ satisfies properties \eqref{axiom 1}--\eqref{axiom 3} of Theorem \ref{Interpolation axiom}. Towards this goal, first we verify this claim for the special case of $r=0$ (and necessarily $n$ even).
Recall that 
\[
W^{\wedge}_{n,0}(\boldsymbol{\alpha}_{[n]}) = \frac{1}{2^{\frac{n}{2}} \left( \frac{n}{2} \right) !} \sum_{ \sigma \in S_{n} }  \sigma \left( \prod_{ 1 \leq i < j \leq \frac{n}{2} } f_{2i-1,2j-1}f_{2i-1,2j}f_{2i,2j-1}f_{2i,2j} \right)
\]
where $f_{i,j} = \frac{(1+\alpha_{i}+\alpha_{j})(\alpha_{i}+\alpha_{j})}{(\alpha_{i}-\alpha_{j})}.$
Applying $\phi_{\Sigma^{\wedge}_{k,0}}$ (as described above) to this expression term-by-term we obtian many 0 terms. The only non-zero terms correspond to $\sigma \in S_{n}$ such that $\sigma = \tau_{1} \dots \tau_{\ell}$ for some $1 \leq \ell \leq \frac{n}{2}$ where $\tau_{i} = (2i'-1,2i')$ for some $1 \leq i' \leq \frac{n}{2}$.
There are $2^{\frac{n}{2}} (\frac{n}{2})!$ such terms, all of the same value, hence
\begin{align*}
\phi_{\Sigma^{\wedge}_{n,0}}(W^{\wedge}_{n,0}) & = \frac{1}{2^{\frac{n}{2}} \left( \frac{n}{2} \right)!} \cdot 2^{\frac{n}{2}} \left( \frac{n}{2} \right)! \prod_{1 \leq i < j \leq \frac{n}{2}} \frac{(1 \pm \sigma_{i} \pm \sigma_{j})(\pm \sigma_{i} \pm \sigma_{j})}{(\pm \sigma_{i} \pm \sigma_{j})} \\
& = \prod_{1 \leq i < j \leq \frac{n}{2}} (1 \pm \sigma_{i} \pm \sigma_{j}) \\
& = c(T_{\Sigma^{\wedge}_{n,0}})=c(T_{\Sigma^{\wedge}_{n,0}})e(N_{\Sigma^{\wedge}_{n,0}}),
\end{align*}
which verifies property \eqref{axiom 1}.

To show that  $W^{\wedge}_{k,0}$ satisfies properties \eqref{axiom 2} and \eqref{axiom 3}, consider the restriction map $\phi_{\Sigma^{\wedge}_{n,m}}$ where $0\leq m \leq n$, $n-m$ even. The non-zero terms in the image $\phi_{\Sigma^{\wedge}_{n,m}}(W^{\wedge}_{n,0})$ are those with $\sigma = \tau_{1} \dots \tau_{\ell} \in S_n$ for some $1 \leq \ell \leq \frac{n-m}{2}$ such that $\tau_{i} = (2i'-1,2i')$ for some $1 \leq i' \leq \frac{n-m}{2}$. The $\phi_{\Sigma^{\wedge}_{n,m}}$-image of $(1 + \alpha_{2i-1} + \alpha_{2j-1})(1 + \alpha_{2i-1} + \alpha_{2j})(1 + \alpha_{2i} + \alpha_{2j-1})(1 + \alpha_{2i} + \alpha_{2j})$ is
\[
\begin{cases}
                 (1 \pm \sigma_{i} \pm \sigma_{j}) & \text{if } \ \ 1 \leq i < j \leq \frac{k-m}{2}, \\
                 (1 \pm \sigma_{i} + \alpha_{2j-1})(1 \pm \sigma_{i} + \alpha_{2j}) & \text{if } \ \ 1 \leq i \leq \frac{k-m}{2} < j \leq \frac{k}{2},
                 \end{cases}
\]
and we see that $\prod_{1 \leq i < j \leq \frac{n-m}{2}} (1 \pm \sigma_{i} \pm \sigma_{j}) \prod_{i=1}^{\frac{n-m}{2}} \prod_{j=n-m+1}^{n} (1 \pm \sigma_{i} + \alpha_{j})=c(T_{\Sigma^\wedge_{n,m}})$ is a common factor in every term of $\phi_{\Sigma^{\wedge}_{n,m}}(W^{\wedge}_{n,0})$, hence $W^{\wedge}_{n,0}$ satisfies property \eqref{axiom 2}.

Assume that $0 < m \leq k$. Each term in the image $\phi_{\Sigma^{\wedge}_{n,m}}(W^{\wedge}_{n,0})$ has degree at most $4 \cdot \binom{n/2}{2}$. Therefore 
\[
\deg(\phi_{\Sigma^{\wedge}_{n,m}}(W^{\wedge}_{n,0})) \leq 4 \cdot \binom{n/2}{2} < \binom{n}{2} - \frac{n-m}{2} = \deg(c(T_{\Sigma^{\wedge}_{n,m}})e(N_{\Sigma^{\wedge}_{n,m}})),
\]
and therefore $W^{\wedge}_{n,0}$ satisfies property \eqref{axiom 3}.

\smallskip

Next we show that the general $W^{\wedge}_{n,r}$ satisfies the properties \eqref{axiom 1}--\eqref{axiom 3} of Theorem \ref{Interpolation axiom}. First consider $\phi_{\Sigma^{\wedge}_{n,r}}(W^{\wedge}_{n,r})$. Due to the factor $\prod_{i \in I} \prod_{j \in \bar{I}} (\alpha_{i}+\alpha_{j})$ in the numerator of $W^{\wedge}_{n,r}$ it follows that only one term has non-zero $\phi_{\Sigma^{\wedge}_{n,r}}$-image, and we obtain
\begin{align*}
\phi_{\Sigma^{\wedge}_{n,r}}(W^{\wedge}_{n,r}) & = \phi_{\Sigma^{\wedge}_{n,r}}(W^{\wedge}_{n-r,0}(\boldsymbol{\alpha}_{\bar{I}})) \prod_{n-r+1 \leq i < j \leq n} (\alpha_{i} + \alpha_{j}) \prod_{i=n-r+1}^{n} \prod_{j=1}^{\frac{n-r}{2}} \frac{(\alpha_{i} \pm \sigma_{j})(1 + \alpha_{i} \pm \sigma_{j})}{(\alpha_{i} \pm \sigma_{j})} \\
& = \prod_{1 \leq i < j \leq \frac{n-r}{2}} (1 \pm \sigma_{i} \pm \sigma_{j}) \prod_{n-r+1 \leq i < j \leq n} (\alpha_{i} + \alpha_{j}) \prod_{i=n-r+1}^{n} \prod_{j=1}^{\frac{n-r}{2}} (1 + \alpha_{i} \pm \sigma_{j}) \\
& = c(T_{\Sigma^{\wedge}_{n,r}}) e(N_{\Sigma^{\wedge}_{n,r}}).
\end{align*}
This proves property \eqref{axiom 1}.

Now let $0 \leq m \leq n$, $n-m$ even, and $m \neq r$, and we study $\phi_{\Sigma^{\wedge}_{n,m}}(W^{\wedge}_{n,r})$. If $m < r$, then $\phi_{\Sigma^{\wedge}_{n,m}}(W^{\wedge}_{n,r})=0$ either because of the factor $\prod_{i < j \in I} (\alpha_{i} + \alpha_{j})$ or because of the factor $\prod_{i \in I} \prod_{j \in \bar{I}} (\alpha_{i} + \alpha_{j})$ in $W^{\wedge}_{n,r}$. If $m > r$ then all the non-zero terms in the image come from the terms in $W^{\wedge}_{n,r}$ with $I$ such that  $I \cap \{1, \dots , n-m\} = \emptyset$. Let $I \subset \{n-m+1, \dots , n\}$, and $\vert I \vert = r$. Then  
\[
\phi_{\Sigma^{\wedge}_{n,m}}(W^{\wedge}_{n-r,0}(\boldsymbol{\alpha}_{\bar{I}})) = \prod_{1 \leq i < j \leq \frac{n-m}{2}} (1 \pm \sigma_{i} \pm \sigma_{j}) \prod_{j \in \bar{I}-[n-m]} \prod_{k=1}^{\frac{n-m}{2}} (1 + \alpha_{j} \pm \sigma_{k}),
\]
and
\begin{align*}
& \phi_{\Sigma^{\wedge}_{n,m}} \left(\prod_{i<j \in I}(\alpha_i + \alpha_j) \prod_{i \in I} \prod_{j \in \bar{I}} \frac{(\alpha_{i}+\alpha_{j})(1+\alpha_{i}+\alpha_{j})}{\alpha_{j}-\alpha_{i}} \right) = \\
& \ \ \ \ \prod_{i<j \in I}(\alpha_i + \alpha_j) \prod_{i \in I} \prod_{j \in \bar{I}-[n-m]} \prod_{k=1}^{\frac{n-m}{2}} \frac{(\alpha_{i} + \alpha_{j})(\alpha_{i} \pm \sigma_{k})(1+\alpha_{i} + \alpha_{j})(1+\alpha_{i} \pm \sigma_{k})}{(\alpha_{j} - \alpha_{i})(\pm \sigma_{k} - \alpha_{i})}.
\end{align*}
The factor $\prod_{1 \leq i < j \leq \frac{n-m}{2}} (1 \pm \sigma_{i} \pm \sigma_{j}) \prod_{j=n-m+1}^{n} \prod_{k=1}^{\frac{n-m}{2}} (1 + \alpha_{j} \pm \sigma_{k})=c(T_{\Sigma^{\wedge}_{n,m}})$ is a common factor in all non-zero terms of $\phi_{\Sigma^{\wedge}_{n,m}}(W^{\wedge}_{n,r})$, which proves property \eqref{axiom 2}.

Now we consider the degree of  $\phi_{\Sigma^{\wedge}_{n,m}}(W^{\wedge}_{n,r})$ for $m \neq r$. We have $\deg(\phi_{\Sigma^{\wedge}_{n,m}}(W^{\wedge}_{n-r,0})) < \binom{n-r}{2} - \frac{n-m}{2}$, and hence
\begin{align*}
\deg(\phi_{\Sigma^{\wedge}_{n,m}}(W^{\wedge}_{n,r})) & <  \binom{n-r}{2} - \frac{n-m}{2}  + \binom{r}{2} + r(n-r) \\
& = \binom{n}{2} - \frac{n-m}{2} \\
& = \deg(c(T_{\Sigma^{\wedge}_{n,m}})e(N_{\Sigma^{\wedge}_{n,m}})),
\end{align*}
proving property \eqref{axiom 3}. This completes the proof of the first statement, $\csm(\Sigma^{\wedge}_{n,r})= W_{n,r}^{\wedge}$, of the theorem. 

The proof of the second statement, the case of $\SCn$, is analogous, we leave it to the reader (or see \cite{P}). 
\end{proof}

\section{Towards the algebraic combinatorics of CSM classes of $\LCn, \SCn$} \label{algcomb}
Characteristic classes of geometrically relevant varieties usually display stabilization and positivity properties. We can expect stabilization properties from the SSM versions, not from the CSM versions, because the SSM version is the one consistent with pull-back (and hence transversal intersection). Also, traditionally the combinatorics of characteristic classes show their true nature when they are expanded in Schur basis.  

Our formulas for the SSM classes obtained in Sections \ref{sec:sieveform} and \ref{sec:intform} can be expanded in Schur basis (to fix our conventions, note that $s_{11}=\sum_{i<j} \alpha_i\alpha_j$, $s_{2}=\sum_{i\leq j} \alpha_i\alpha_j$), and we obtain

\begin{align*}
\ssm(\Sigma^\wedge_{2,0})= & s_0 - s_1 + (s_{2} + s_{11}) - (s_3+2s_{21}) + (s_4+2s_{22}+3s_{31}) - \ldots \\
\ssm(\Sigma^\wedge_{4,0})= & s_0 - s_1 + (s_{2} + s_{11}) - (s_3+2s_{21}+s_{111}) + (s_4+2s_{22}+3s_{31}+3s_{211}+s_{1111}) \\
 &  \ \hskip 4.5 true cm - (s_5+5s_{32}+4s_{41}+5s_{221}+6s_{311}+4s_{2111})+\ldots \\
\ssm(\Sigma^\wedge_{2,2})= & s_1 - (s_2+s_{11}) + (s_3+2s_{21}) - (s_4+2s_{22}+3s_{31}) + \ldots \\
\ssm(\Sigma^\wedge_{4,2})= & s_1 - (s_2+s_{11}) + (s_3+2s_{21}+s_{111}) - (s_4+2s_{22}+3s_{31}+3s_{211}+s_{1111}) \\
& \hskip 4.7 true cm  + (s_5+5s_{32}+4s_{41}+5s_{221}+6s_{311}+4s_{2111}) - \ldots \\
\ssm(\Sigma^\wedge_{4,4}) = & s_{321} -  (3s_{322}+3s_{331}+3s_{421}+3s_{3211}) \\
& \hskip 1 true cm + (10s_{332}+10s_{422}+10s_{431}+6s_{521}+10s_{3221}+10s_{3311} +10s_{4211}) - \ldots
\end{align*}

\begin{align*}
\ssm(\Sigma^S_{1,0})=&  s_0-2s_1+4s_2-8s_3+16s_4-\ldots  \\
\ssm(\Sigma^S_{2,0})=&  s_0  - 2s_1 + (4s_2 + 4s_{11}) - (8s_3 + 12s_{21}) + (16s_4 + 12s_{22} + 28s_{31}) - \dots \\
\ssm(\Sigma^S_{3,0})=&  s_0 - 2s_1 + (4s_2 + 4s_{11}) - (8s_3 + 12s_{21} + 8s_{111}) + (16s_4 + 12s_{22} + 28s_{31} + 28s_{211}) - \dots \\
\ssm(\Sigma^S_{1,1})=&  2s_1-4s_2+8s_3-16s_4+\ldots  \\
\ssm(\Sigma^S_{2,1})=&   2s_1 - (4s_2+4s_{11}) + (8s_3 + 8s_{21}) - (16s_4 + 16s_{31}) + \dots \\
\ssm(\Sigma^S_{3,1})=&  2s_1 - (4s_2+4s_{11}) + (8s_3 + 8s_{21} + 8s_{111}) - (16s_4 + 16s_{31} + 16s_{211}) + \dots\\
\ssm(\Sigma^S_{2,2})=&  4s_{21} - (12s_{22}+12s_{31}) + (40s_{32}+28s_{41}) - \ldots  \\
\ssm(\Sigma^S_{3,2})=&  4s_{21} - (12s_{22} + 12s_{31} + 12s_{211}) + (40s_{32} + 28s_{41} + 40s_{221} + 40s_{311}) - \dots\\
\end{align*}

Here are some observations on these expressions:
\begin{itemize}
\item {\bf (stabilization)} For $n<m$ the formula for $\ssm(\Sigma_{m,r})$ also works for $\ssm(\Sigma_{n,r})$ (either $\LCn$ or $\SCn$). Of course, some $s_\lambda$ functions may be non-0 for $m$ variables, but $0$ for $n$ variables, so some terms of $\ssm(\Sigma_{m,r})$ are not necessary to name $\ssm(\Sigma_{n,r})$. The reason for this stabilization, on the one hand,  is that $\LCn$ can be viewed as a linear section of $\LCm$ (consistent with the group actions) such that orbits of $\LCn$ are transversal to this linear space, and on the other hand, SSM classes are consistent with transversal intersection, see~(iv) in Section \ref{sec:eqCSM}. Due to this stabilization one may consider the limit objects (formal power series) $\ssm(\Sigma_{\infty,r})$. In \cite{P} generating functions (in the ``iterated residue'' sense) are presented for these limit power series.
\item {\bf (positivity)} We expect that the coefficients of SSM classes in Schur basis have predictable signs. All the above examples support 
\begin{conjecture} The Schur expansions of $\ssm(\Sigma^\wedge_{n,r})$ and $\ssm(\Sigma^S_{n,r})$ have alternating signs. \end{conjecture}
\noindent The sign behavior of SSM classes is determined under very general circumstances in \cite{AMSSpos}. It would be interesting to check whether results in that paper imply our conjecture. 
\item {\bf (lowest degree terms)} The lowest degree terms of both CSM and SSM class is the fundamental class of the closure of the orbit. Hence the expressions above are all of the form $\ssm(\Sigma^\wedge_{n,r})=s_{r-1,r-2,\ldots,1}+h.o.t.$, $\ssm(\Sigma^S_{n,r})=2^{r-1}s_{r,r-1,\ldots,1}+h.o.t.$. 
\item {\bf (normalization)} The additivity property of CSM classes imply that the sum of the SSM classes of all orbits of a representation (with finitely many orbits) is 1---we encourage the reader to verify this property in the above examples. Normalization and positivity properties together indicate a formal similarity between SSM theory and probability theory, cf. \cite[Remark 8.8]{FR2}.
\end{itemize}

Our choice above for expanding in terms of {\em Schur functions}, is essentially due to tradition. Schur functions are the fundamental classes of so-called matrix Schubert varieties. Those varieties are indeed very basic ones, but the choice of choosing their {\em fundamental class} as our basic polynomials might be improved sometimes. It might be more natural that for SSM classes of geometrically relevant varieties the ``right'' choice of expansion is in terms of the {\em SSM classes} of matrix Schubert varieties. These functions, named $\st_\lambda$, are defined and calculated in \cite[Definition 8.2]{FR2}. Moreover, the $\st_\lambda=s_\lambda+h.o.t.$ functions are themselves Schur alternating (conjectured in \cite{FR2}, proved in \cite{AMSSpos}). Remarkably, SSM classes of some quiver loci are proved to be $\st_\lambda$-positive---indicating a two-step positivity structure of SSM classes (see the Introduction of \cite{FR2}). The following are $\st_\lambda$-expansions:

\begin{align*}
\ssm(\Sigma^{\wedge}_{\infty,0})=&
\st_0 + \st_{22} + (\st_{44} + \st_{2222}) + (\st_{66} + \st_{4422} + \st_{222222}) + \dots, \\
\ssm(\Sigma^{\wedge}_{\infty,1})=& \st_0 + \st_1 + (\st_2 + \st_{11}) + (\st_3 + \st_{111}) + (\st_4 + \st_{1111}) + (\st_5 + \st_{11111}) \\ 
& + (\st_6 + \st_{33} + \st_{222} + \st_{111111}) + \dots \\
\ssm(\Sigma^{\wedge}_{\infty,2})=&
\st_1 + (\st_2 + \st_{11}) + (\st_3 + \st_{21} + \st_{111}) + (\st_4 + \st_{31} + \st_{211} + \st_{1111}) \\
& + (\st_5 + \st_{41} +  \st_{32} + \st_{311} + \st_{221} + \st_{2111} + \st_{11111}) + \dots, \\
\end{align*}
\begin{align*}
\ssm(\Sigma^S_{\infty,0}) &= \st_0 - \st_1 + (\st_2 + \st_{11}) - (\st_3 + \st_{21} + \st_{111}) + (\st_4 + \st_{31} + \st_{22} + \st_{211} + \st_{1111}) \\
& - (\st_5 + \st_{41} + \st_{32} + \st_{311} + \st_{221} + \st_{2111} + \st_{11111}) + \dots, \\
\ssm(\Sigma^S_{\infty,1}) &= 2\st_1 + (2\st_3 - 2\st_{21} + 2\st_{111}) + (2\st_5 - 2\st_{41} + 2\st_{32} + 2\st_{311} + 2\st_{221} - 2\st_{2111} + 2\st_{11111}) \\ 
&- 4\st_{321} + \dots, \\
\ssm(\Sigma^S_{\infty,2}) &= 4\st_{21} + (4\st_{41} + 4\st_{2111}) - 4\st_{321} + (4\st_{61} + 4\st_{43} + 4\st_{4111} + 4\st_{2221} + 4\st_{211111}) \\ 
& - (4\st_{521} + 4\st_{32111}) + \dots, \\
\ssm(\Sigma^S_{\infty,3}) &= 8\st_{321} + (8\st_{521} + 8\st_{32111}) + (8\st_{721} + 8\st_{541} - 8\st_{4321} + 8\st_{32221} + 8\st_{3211111}) \\ 
& + (8\st_{921} + 8\st_{741} -8\st_{6321} + 8\st_{543} + 8\st_{33321} - 8\st_{432111} + 8\st_{3222111} + 8\st_{321111111}) + \dots, \\
\ssm(\Sigma^S_{\infty,4}) &= 16\st_{4321} + (16\st_{6321} + 16\st_{432111}) + \dots
\end{align*}

It is worth verifying in these examples the normalization properties 
\[
\sum_i \ssm(\Sigma^S_{\infty,i})=
\sum_i \ssm(\Sigma^\wedge_{\infty,2i})=
\sum_i \ssm(\Sigma^\wedge_{\infty,2i+1})=
\sum_\lambda \st_\lambda=1.
\]

The calculated $\st_\lambda$-expansions (the ones above and many more) display several patterns; let us phrase two of them as conjectures.

\begin{conjecture}
\begin{itemize} 
\item The $\st_\lambda$-expansions of $\ssm(\Sigma^\wedge_{\infty,i})$ and $\ssm(\Sigma^S_{\infty,i})$ are invariant under $\lambda\mapsto \lambda^T(=$the transpose partition, eg. $(6321)^T=432111$). That is, in both expansions the coefficient of $\st_\lambda$ and the coefficient of $\st_{\lambda^T}$ are the same.
\item The $\st_\lambda$-expansion of $\ssm(\Sigma^\wedge_{\infty,i})$ has non-negative coefficients. The $\st_\lambda$-expansion of $\ssm(\Sigma^\wedge_{\infty,2i})$ have alternating coefficients.
\end{itemize}
\end{conjecture}

\section{Applications}\label{appl}

We will apply the calculated CSM classes to find the Euler characteristics of general linear sections of the projectivizations of $\Sigma^\wedge_{n,r}$, $\Sigma^S_{n,r}$. First we study the relation between characteristic classes in vector spaces and in projective spaces.

\subsection{Characteristic classes before vs after projectivization} 
Consider the algebraic representation $G\acts V=\C^N$, with $T=(\C^*)^m\leq G$ the maximal torus, and weights $\sigma_j$. Suppose the representation contains the scalars, that is, there is a map $\phi:\C^*\to T$, $\phi(s)=(s^{w_1},s^{w_2},\ldots,s^{w_m})$ such that $\phi(s)$ acts on $V$ with multiplication by $s^w$ ($w\not=0$). Then the $G$-invariant subsets $\Sigma\subset V$ are necessarily cones. 

We want to compare the $G$-equivariant characteristic classes of $\Sigma \subset V$ with those of $\PS\subset \PV$. The first one lives in the ring $H_G^*(V)=H^*(BG)$, while the second one lives in
\begin{equation}\label{Hproj}
H_G^*(\PV)=H^*_G(\PV)=H^*(BG)[\xi]\ /\ \textstyle{\prod_j} (\xi-\sigma_j),
\end{equation}
where $\xi$ is the first Chern class of the $G$-equivariant tautological line bundle over $\PV$. Here $H^*(BG)$ is a subring of $H^*(BT)=\Q[\alpha_1,\ldots,\alpha_m]$, and the weights $\sigma_j$ are linear combinations of the $\alpha_i$'s.

\begin{theorem}
The substitutions
\[
[\Sigma]|_{\alpha_i\mapsto \alpha_i+\frac{w_i}{w}\xi},\qquad\qquad
\ssm(\Sigma)|_{\alpha_i\mapsto \alpha_i+\frac{w_i}{w}\xi}
\]
represent $[\PS\subset \PV]$ and $\ssm(\PS\subset \PV)$, respectively, in \eqref{Hproj}. 
\end{theorem}

The first statement is \cite[Theorem 6.1]{FNR}, and the proof there holds for SSM classes as well, because the proof given there only uses the pull-back property which $[\ \ ]$ shares with SSM classes.

The non-equivariant (``ordinary'') characteristic classes are always obtained from the equivariant ones by substituting 0 in the equivariant variables. Therefore, the non-equivariant SSM class $\ssmz(\PS\subset \PV)$ of $\PS$ living in $H^*(\PV)=\Q[\xi]/\xi^{N}$ is obtained as
\[
\ssmz(\PS\subset \PV)=\ssm(\Sigma\subset V)|_{\alpha_i\mapsto \frac{w_i}{w}\xi}.
\]
Remarkably, the same holds for CSM classes too: $\csmz(\PS\subset \PV)=\csm(\Sigma\subset V)|_{\alpha_i\mapsto \frac{w_i}{w}\xi}$, which follows from the calculation 
\begin{align*}
\csmz(\PS) &= \ssmz(\PS) c(\PV) = \ssm(\Sigma)|_{\alpha_i\mapsto \frac{w_i}{w}\xi} \cdot (1+\xi)^N \\
&= \left.\frac{\csm(\Sigma)}{\prod_j(1+\sigma_j)}\right|_{\alpha_i \mapsto \frac{w_i}{w}\xi} \cdot (1+\xi)^N 
= \csm(\Sigma)|_{\alpha_i\mapsto \frac{w_i}{w}\xi}, 
\end{align*}
where the last equality used the defining property of $w_i$, $w$, namely: $\left. 1 + \sigma_j \right|_{\alpha_i \mapsto \frac{w_i}{w} \xi} = 1 + \xi$ 
(for all $j$).

\subsection{Non-equivariant CSM classes of symmetric and skew-symmetric determinantal varieties}\label{sec:noneq}
Let us study the projectivizations of $\Sigma^\wedge_{n,r}$ and $\Sigma^S_{n,r}$. Some of these projective varieties are well known: $\PP\Sigma^\wedge_{n,n-2}$ is the Pl\"ucker embedding of $\Gr_2\C^n$, and $\PP\Sigma^S_{n,n-1}$ is the Veronese embedding of $\PP\C^n$. Applying the result of the preceding section to $\LCn$ and $\SCn$ we find that the ordinary CSM classes of the orbits are obtained by
\begin{align*}
\csmz(\PP\Sigma^\wedge_{n,r} \subset \PP\LCn) & = \csm(\Sigma^\wedge_{n,r} \subset \LCn)|_{\alpha_i\mapsto \xi/2}, \\
\csmz(\PP\Sigma^S_{n,r} \subset \PP\SCn) & = \csm(\Sigma^S_{n,r} \subset \SCn)|_{\alpha_i\mapsto \xi/2}.
\end{align*}

For example from the explicit formula for $\csm(\Sigma^S_{3,r}\subset S^2\!\C^3)$ given in Definition \ref{def local symm weight fcn}, after substituting $\alpha_i=\xi_i/2$ for all $i$ we obtain
\begin{align}\label{exS3}
\notag \csmz(\PP\Sigma^S_{3,0}) = & 1+ 3\xi + 6\xi^2+\ \ 6\xi^3+3\xi^4, \\
          \csmz(\PP\Sigma^S_{3,1}) = & \hskip .75 true cm  3\xi+9\xi^2+10\xi^3+6\xi^4+3\xi^5, \\
\notag \csmz(\PP\Sigma^S_{3,2}) = & \hskip 3 true cm 4\xi^3+6\xi^4+3\xi^5.
\end{align}
As we know, the degrees of the lowest degree terms above, namely $0, 1, 3$, are the codimensions of the given orbits. The coefficients of the lowest degree terms, namely $1, 3, 4$, are the degrees of the closures of those orbits. The integral, ie. the coefficients of $\xi^5$, namely $0, 3, 3$ are the Euler charactersitics of the orbits. Observe also, that the sum of the three classes above are $1+6\xi+15\xi^2+20\xi^3+15\xi^4+6\xi^5=c(T\PP^5)=(1+\xi)^6\in \Q[\xi]/(\xi^6)$.

The codimensions, the degrees, and the Euler characteristics of the orbits of $\PP\LCn$ and $\PP\SCn$ are 
\begin{align*}
\codim(\PP\Sigma^\wedge_{n,r}\subset \PP\LCn)=\binom{r}{2},& \qquad
\codim(\PP\Sigma^S_{n,r}\subset \PP\SCn)=\binom{r+1}{2},
\\
\deg(\PP\Sigma^\wedge_{n,r}\subset \PP\LCn)=\frac{1}{2^{r-1}}\prod_{i=0}^{r-2}\frac{\binom{n+i}{r-1-i}}{\binom{2i+1}{i}},
& \qquad
\deg(\PP\Sigma^S_{n,r}\subset \PP\SCn)=\prod_{i=0}^{r-1}\frac{\binom{n+i}{r-i}}{\binom{2i+1}{i}},
\\
\chi(\PP\Sigma^\wedge_{n,r}) = \begin{cases} \binom{n}{2}  & \text{for } \ \ r=n-2 \\
0 & \text{for } \ \ r<n-2, 
\end{cases}
& \qquad
\chi(\PP\Sigma^S_{n,r}) = \begin{cases} n & \text{for } \ \ r=n-1 \\
\binom{n}{2} & \text{for } \ \ r=n-2 \\
0 & \text{for } \ \ r<n-2.
\end{cases}
\end{align*}
These formulas are well known, and can be found either by classical methods or using our formula, see some details in \cite{P}. The question is: what geometric information is carried by the `middle' terms of the $\csmz$ classes like~\eqref{exS3}. This will be answered in the next section.

\subsection{Euler characteristics of general linear sections}
Let $X\subset \PP^N$ be a locally  closed set, and let $X_r=X\cap H_1 \cap \ldots \cap H_r$ be the intersection with $r$ general hyperplanes. Following \cite{Al} define the {\em Euler characteristic polynomial of $X$} to be
\[
\chi_X(t)=\sum_{i=0}^N \chi(X_i)(-t)^i.
\]
From the non-equivariant CSM class of $X$, $\csmz(X\subset \PP^N)=\sum_{i=0}^N a_i \xi^i$ define $\gamma_X(t)=\sum_{i=0}^N a_i t^{N-i}$.
Aluffi showed that the two polynomials $\chi_X$ and $\gamma_X$ are related as follows. For a polynomial $p(t)$ define 
\[
\J(p)(t) =\frac{t p(-t-1) + p(0)}{t+1}.
\]
The operation $\J$ is a degree-preserving linear involution on polynomials in $t$.

\begin{theorem}\cite[Theorem 1.1]{Al} \label{aluffi}
For every locally closed subset $X$ of $\mathbb{CP}^N$, we have
\[
\mathcal{J}(\chi_X(t)) = \gamma_X(t) \ \ \ \ \text{and} \ \ \ \ \mathcal{J}(\gamma_X(t)) = \chi_X(t).
\]
\end{theorem}

Putting our results together with Theorem \ref{aluffi} we have an algorithm to find the Euler characteristics of general linear sections of the orbits of $\PP\Sigma^\wedge_{n,r}$, $\PP\Sigma^S_{n,r}$. Namely: Formulas for $\GL_n(\C)$-equivariant CSM classes of  $\Sigma^\wedge_{n,r}$, $\Sigma^S_{n,r}$ are given in Sections \ref{sec:sieveform} and \ref{sec:intform}. Those formulas turn to formulas for non-equivariant CSM classes for $\PP\Sigma^\wedge_{n,r}$, $\PP\Sigma^S_{n,r}$ in Section~\ref{sec:noneq}. According to Theorem \ref{aluffi}, the coefficients of the $\J$-operation of those non-equivariant CSM classes are the Euler characteristics of the general linear sections.

\begin{example} \rm
From the calculations in \eqref{exS3} we get 
\begin{align*}
\gamma_{\mathbb{P}\Sigma^{S}_{3,0}}(t) &= 3t+6t^2+6t^3+3t^4+t^5, \\
\gamma_{\mathbb{P}\Sigma^{S}_{3,1}}(t) &= 3+6t+10t^2+9t^3+3t^4, \\
\gamma_{\mathbb{P}\Sigma^{S}_{3,2}}(t) &= 3+6t+4t^2.
\end{align*}
After applying the involution $\mathcal{J}$ we get the Euler characteristic polynomials
\begin{align*}
\chi_{\mathbb{P}\Sigma^{S}_{3,0}}(t) &= (-t)-(-t)^2+3(-t)^3-(-t)^4+(-t)^5, \\
\chi_{\mathbb{P}\Sigma^{S}_{3,1}}(t) &= 3+2(-t)+(-t)^2+3(-t)^4, \\
\chi_{\mathbb{P}\Sigma^{S}_{3,2}}(t) &= 3+2(-t)+4(-t)^2,
\end{align*}
and the Euler characteristics of Table \ref{table1}.  
Similar calculation yields e.g. the Euler characteristics presented in Table \ref{table2}. 

It is worth verifying that the sum of columns (in both tables) is the Euler characteristic of the appropriate projective linear space.

\begin{table}[h!]
\centering
\begin{tabular}{ |c|c|c|c|c|c|c| } 
 \hline $X$ & $\chi(X)$ & $\chi(X_1)$ &$\chi(X_2)$ & $\chi(X_3)$ & $\chi(X_4)$ & $\chi(X_5)$ \\ 
 \hline $\mathbb{P}\Sigma^{S}_{3,0}$ & 0 & 1 & -1 & 3 & -1 & 1 \\ 
 \hline $\mathbb{P}\Sigma^{S}_{3,1}$ & 3 & 2 & 1 & 0 & 3 & 0 \\
 \hline $\mathbb{P}\Sigma^{S}_{3,2}$ & 3 & 2 & 4 & 0 & 0 & 0 \\
 \hline
\end{tabular}
\caption{Euler characteristics of general linear sections of the orbits in $\PP S^2\!\C^3$}\label{table1}
\end{table}

\begin{table}[h!]
\centering
\begin{tabular}{ |c|c|c|c|c|c|c|c|c| } 
 \hline $X$ & $\chi(X)$ & $\chi(X_1)$ &$\chi(X_2)$ & $\chi(X_3)$ & $\chi(X_4)$ & $\chi(X_5)$ & $\chi(X_6)$ & $\chi(X_7)$ \\ 
 \hline $\mathbb{P}\Sigma^{\wedge}_{6,0}$ & 0&-1&1&-3&5&-11&21&-29 \\ 
 \hline $\mathbb{P}\Sigma^{\wedge}_{6,2}$ & 0&3&0&9&-6&27&-36&51 \\
 \hline $\mathbb{P}\Sigma^{\wedge}_{6,4}$ & 15&12&12&6&12&-6&24&-14 \\
 \hline
\end{tabular}

\smallskip

\begin{tabular}{ |c|c|c|c|c|c|c|c| } 
 \hline $X$ & $\chi(X_8)$ & $\chi(X_9)$ &$\chi(X_{10})$ & $\chi(X_{11})$ & $\chi(X_{12})$ & $\chi(X_{13})$ & $\chi(X_{14})$  \\ 
 \hline $\mathbb{P}\Sigma^{\wedge}_{6,0}$ & 29&-21&11&-5&3&-1&1 \\ 
 \hline $\mathbb{P}\Sigma^{\wedge}_{6,2}$ & -36&27&-6&9&0&3&0 \\
 \hline $\mathbb{P}\Sigma^{\wedge}_{6,4}$ & 14&0&0&0&0&0&0 \\
 \hline
\end{tabular}
\caption{Euler characteristics of general linear sections of the orbits in $\PP\Lambda^2\!\C^6$} \label{table2}
\end{table}
\end{example}

\section{Future directions}\label{future}

\subsection{Chern-Mather classes}
Another reason for studying CSM classes of singular varieties is the relation with their Chern-Mather classes. For the role of Chern-Mather classes in geometry see the recent paper \cite{AMather} and references therein. One approach to Chern-Mather classes is the construction called Nash blow-up, another one is the natural transformation $C_*:\mathcal F^G(-)\to H_*^G(-)$ mentioned in Remark \ref{OhmotosDef}. As we know, the (homology) CSM class of a closed subvariety $W$ is $C_*(\One_W)$. The Chern-Mather class $\cM(W)$ of $W$ is the $C_*$-image of another remarkable constructible function, the so-called local Euler obstruction function, $\Eu_W$. Hence if $\Eu_W$ can be calculated, ie. expressed as a linear combination of $\One_{V_i}$'s (for locally closed set $V_i$), then the same linear relation holds among $\cM(W)$ and the CSM classes of $V_i$'s. Arguments along these lines are carried out in \cite{P} resulting the following theorem.

\begin{theorem}{\rm \cite[Thms 9.11, 9.18]{P}}
For $0 \leq r \leq n$ we have
\[
\Eu_{\overline{\Sigma}^{\wedge}_{n,r}} = \sum_{k=0}^{\frac{n-r}{2}} \binom{\lfloor \frac{r}{2} \rfloor + k}{\lfloor \frac{r}{2} \rfloor} \One_{\Sigma^{\wedge}_{n,r+2k}},
\quad\text{and hence}\quad
\cM(\overline{\Sigma}^{\wedge}_{n,r}) = \sum_{k=0}^{\frac{n-r}{2}} \binom{\lfloor \frac{r}{2} \rfloor + k}{\lfloor \frac{r}{2} \rfloor} \csm(\Sigma^{\wedge}_{n,r+2k}).
\]
\end{theorem}

The authors do not know the local Euler obstructions for the orbit closures in $\SCn$.

\subsection{K theory generalization: motivic Chern classes} There is a natural generalization of the cohomological notion of CSM class to K theory, called {\em motivic Chern class}. It was defined in \cite{BSY} and the equivariant version is set up in \cite{FRW2}. 

The equivariant motivic Chern class $\mC(\Sigma)$ of an invariant subvariety $\Sigma\subset M$ of the smooth ambient variety $M$ lives in $K_G(M)[y]$. The $y=1$ specialization recovers the K theory fundamental class. It is convenient to consider its Segre version, the motivic Segre class $\mS(\Sigma)=\mC(\Sigma)/c(TM)$ where $c(TM)$ is the K theoretic total Chern class. Hence, $\mC$ and $\mS$ of the orbits of $\LCn$ and $\SCn$ are elements of (a completion of)
\[
K_{\GL_n(\C)}(\text{pt})=\Z[\alpha_1^{\pm 1},\ldots,\alpha_n^{\pm 1}]^{S_n}[y],
\]
where $\alpha_i$ are the K theory Chern roots $\GL_n(\C)$, ie. their sum is the tautological $n$-bundle over $B\!\GL_n(\C)$.

The traditional approach to study motivic Chern classes is through resolutions and a property similar to \eqref{csm property 3} (with the notion of Euler characteristic replaced with the notion of chi-y-genus). Our construction in Section \ref{fibred} fits that approach, and hence arguments analogous to those in Section \ref{sec:sieveLambda} can be carried out to obtain a sieve formula for the motivic Segre class $\mS$ of the orbits of $\LCn$, $\SCn$. 

\begin{theorem} {\rm \cite[Cor. 10.15]{P}} \label{Kth}
Let $0\leq r\leq n$, $n-r$ even, and $q=-h$. We have
\[
\mS(\Sigma^\wedge_{n,r}) = \sum_{k=0}^{\frac{n-r}{2}} \binom{r+2k}{r}_q E_{2k}(q) \Phi^\wedge_{n,r+2k},
\]
where 
\[
\Phi^{\wedge}_{n,r} = \sum_{\substack{ I \subset [n] \\ \mid I \mid = r }} \left( \prod_{i<j \in I}\frac{1-\frac{1}{\alpha_i\alpha_j}}{1+\frac{y}{\alpha_i\alpha_j}} \prod_{i \in I} \prod_{j \in \bar{I}} \frac{\left(1-\frac{1}{\alpha_i\alpha_j}\right)\left(1+\frac{y\alpha_j}{\alpha_i}\right)}{\left(1+\frac{y}{\alpha_i\alpha_j}\right)\left(1-\frac{\alpha_j}{\alpha_i}\right)} \right),
\]
\[
\binom{n}{m}_q = \frac{[n]_q!}{[m]_q![n-m]_q!}, \qquad\qquad [n]_q! = [1]_q[2]_q \dots [n]_q, \qquad\qquad [0]_q! = 1,
\]
and the $q$-Euler numbers $E_n(q)$ are defined by 
\[
\frac{1}{\cosh_q(t)} = \sum_{n=0}^{\infty}\frac{E_n(q)}{[n]_q!}\cdot t^n, \qquad \cosh_q(t) = \sum_{i=0}^{\infty}\frac{t^{2n}}{[2n]_q!}.
\]
\end{theorem}

The $y=1$ specialization recovers the K theory fundamental classes (for more works on those classes see \cite{And}). Yet, Theorem~\ref{Kth} is just a (rather complicated) sieve formula. The desired formula would be analogous to the interpolation formula Theorem \ref{localization thm} for CSM classes. Although interpolation characterization of motivic Chern classes exist \cite[Section 5.2]{FRW2}, the solution of those interpolation constraints (involving Newton polytopes of specializations) is highly non-trivial, and hence is subject to future study. Initial results and conjectures are in \cite{P}.




\end{document}